\documentclass[12pt]{amsart}

\usepackage{dsfont}
\usepackage{amsmath}
\usepackage{amssymb}
\usepackage{graphicx}

\usepackage{amssymb}

\usepackage{amsfonts}

\usepackage{soul}

\usepackage{mathpazo}
\usepackage{color}
\usepackage{yfonts}

\usepackage{paralist}

\usepackage{stmaryrd}

\usepackage{amsxtra}

\def\carre{\hfill $\Box$}
\def\id{{\rm Id}}
\def\tf{\tilde{f}}
\def\tg{\tilde{g}}

\def\td{\widetilde{\Delta}}        
\def\cU{\mathcal U}
\def\tI{\tilde{I}}
\def\tN{\tilde{N}}
\def\bC{{\bar C}}
\def\cA{\mathcal A}
\def\eps{\varepsilon}

\def\ci{     \circ}

\def\b{          \beta}

\def\a{          \alpha}

\def\cA{          \mathcal A}
\def\cD{          \mathcal D}

\def \vp{{\varphi}}

\def \si{{\sigma}}

\def \R{{\mathbb R}}
\def \la{{\lambda}}
\def \Z{{\mathbb Z}}
\def \N{{\mathbb N}}
\def \t{{\theta}}
\def \l{{\lambda}}

\def \ci{{C^\infty}}

\newcommand{\T}{{\mathbb T}}
\newcommand{\prf}{{\begin{proof}}}
\newcommand{\epf}{{\end{proof}}}

\newtheorem{theo}{\sc Theorem}
\newtheorem{prop}{\sc Proposition}
\newtheorem{lemm}{\sc Lemma}

\theoremstyle{definition}

\def\bee{\begin{equation}}
\def\eee{\end{equation}}
\def\f{\varphi}
\def\p{\psi}

\theoremstyle{remark}

\newcommand{\pdvr}[2]
{\dfrac{\partial^{#2} #1}{\partial \theta^{#2_1} \partial r^{#2_2}}}
\newcommand{\pdvrs}[2]
{\partial^{#2} #1 /\partial \theta^{#2_1} \partial r^{#2_2}}
\newcommand{\ft}{{f_t}}

\newcommand{\RR}{{\mathbb R}}

\newcommand{\ZZ}{{\mathbb Z}}

\theoremstyle{definition}

\newtheorem{definition}{ Definition}
\numberwithin{equation}{section}

\begin{document}

\title[KAM rigidity for partially hyperbolic affine $\Z^k$ actions]{KAM rigidity for partially hyperbolic affine $\Z^k$ actions on the torus with a rank one factor}
\author[Danijela Damjanovi\'c and Bassam Fayad]{Danijela Damjanovi\'c$^1$  and Bassam Fayad $^2$}
\thanks{ $^1$ Based on research supported by NSF grant  DMS-0758555 }


\subjclass[2012]{}


\address{Department of Mathematics, \\
      Rice University\\6100 Main st\\Houston, TX 77005}
\email{dani@rice.edu\\}


\begin{abstract} We show that ergodic affine $\mathbb Z^k$ actions on the torus, that have a rank-one factor in their linear part, are locally rigid in a KAM sense if and only if the rank one factor is trivial and the action is higher-rank transversally to this factor. 
Since \cite{DK} proves that affine actions with higher-rank linear part are locally rigid, our result completes the local rigidity picture for affine actions on the torus.

\end{abstract}

\maketitle

\section{Introduction}

\subsection{Local rigidity of $\Z^k$ actions} 

A smooth  $\Z^k$ action $\rho$ on a smooth manifold $M$ is said to be locally rigid if there exists a neighborhood $\mathcal U$ of $\rho$ in the space of smooth  $\Z^k$ actions on $M$, such that for every $\eta\in \mathcal U$ there is a $C^\infty$ diffeomorphism $h$ of $M$ such that $h\circ \rho\circ h^{-1}=\eta$. 

When $k=1$ or if the $\Z^k$ action has a factor which is (perhaps up to a finite index subgroup) an action of $\mathbb Z$, then one cannot expect to have local rigidity as described above. The only known situation in rank-one dynamics where some form of  local rigidity happens is for Diophantine toral translations, where translation vectors with respect to invariant probability measures serve as moduli. If a Diophantine translation is perturbed into a parametric family of diffeomorphisms and if the translation vectors relative to invariant measures satisfy an adequate transversality condition, then for a large set of parameters the diffeomorphisms of the family are smoothly conjugate to translations. This is a consequence of KAM theory (after Kolmogorov, Arnol'd and Moser) and we call it {\it KAM rigidity}.   A typical example is given by Arnol'd family of circle diffeomorphisms \cite{A} where transversality { in this case} amounts to the requirements that the rotation number of the diffeomorphisms should often be Diophantine. The latter example { will be a paradigm }  in the subsequent study of partially hyperbolic affine actions.

{ 
\subsection{Local rigidity of higher rank actions by automorphisms of the torus} }



A $\Z^k$, $k\ge 2$  action which has no rank-one factors is called a  \emph{genuinely} higher-rank action, or just a higher-rank action. 

For ergodic actions by toral automorphisms it is proved in \cite{Starkov} that the action has no rank-one factors if and only if: 

{\it (HR) The  $\Z^k$  action contains a subgroup $L$ isomorphic to $\Z^2$  such that every element in $L$, except for identity, acts ergodically with respect to  the standard invariant measure obtained
from Haar measure.} 

This condition may be viewed as a general paradigm for any form of rigidity of an algebraic action. 

Notice that the condition (HR) for a $\Z^k$  action by toral automorphisms implies that the action is partially hyperbolic, since ergodicity for a single toral automorphism immediately implies partial hyperbolicity.

The local picture for ergodic higher-rank $\Z^k$ actions on the torus by toral automorphisms is fairly well understood. The condition \emph{(HR)}  is a necessary and sufficient condition for local rigidity ( \cite{DK} and references therein). 



The { local rigidity} result in \cite{DK} extends to \emph{affine} actions on the torus whose linear parts are actions which satisfy the (HR) condition.

{ The specificity of affine actions }appears nevertheless if the  linear part violates the assumption (HR). 
For example take the $\mathbb Z^2$ action on $\T^{d+1}$ generated by $A \times \text {Id}$,  $B \times \text {Id}$, with $A$ and $B$  two hyperbolic commuting toral automorphisms of $\T^d$. Of course this $\Z^2$ action  does not satisfy the ergodicity assumption in the general paradigm. 
But in the affine setting, the $\mathbb Z^2$ action generated by  $A \times R_\a$ and  $B \times R_\beta$, where $R_\a$ and $R_\beta$ are two circle rotations such that $1,\a,\beta$ are rationally independent,  satisfies the ergodicity assumption of the general paradigm, while its linear part does not. This action is clearly  not locally rigid. We can for example change the frequencies, but even with fixed frequencies, Anosov-Katok  Liouville constructions show that we can perturb $R_\a \times R_\beta$ into a non linearizable commuting pair { of circle diffeomorphisms.} 

In this paper we consider affine actions on the torus  which have as their linear part a $\mathbb Z^k$ action which does not satisfy (HR). We show that for such actions { certain kind of local rigidity} may be established if and only if { there exists a set of generators}  of the linear part  given by  $A_i \times \text {Id}$ where $A_1,\ldots,A_k$ satisfy (HR).

Since the statements for $\Z^k$ actions are exactly similar to the ones for  $\Z^2$ actions, we state our results in the latter case for better readability of the results and the proofs.

It is easy to see that local rigidity of an affine action $\rho$ whose linear part does not satisfy (HR) can only be possible if its generators, after a coordinate change,   are of the form 
\begin{equation} \label{normalform} \bar A=(A+a) \times (\text {Id}_{\T^{d_2}}+\f),  \bar B= (B+b) \times (\text {Id}_{\T^{d_2}}+\p)\end{equation}  with $A$ and $B$  two 
commuting toral automorphisms of $\T^{d_1}$  that satisfy  (HR), where $d_1+d_2=d$, and   $a,b,\f, \psi$ are translation vectors. Indeed, if the action generated by the generators  of the linear part of $\rho$ has a rank-one factor then up to a coordinate change in $\Z^2$ we may assume that 
$\bar A=(A+a) \times (Id+\f)$ and $\bar B= (B+b) \times (C+\p)$, where  $A$ and $B$ generate  a linear action, $(\f, \psi)$ are translation vectors, and $C$ is a linear map. The commutativity condition  implies that   $\f$ belongs to the eigenspace $V_1$ of $C$ relative to the eigenvalue $1$. If $C \neq \text{Id}$, { projecting to the orthocomplement of $V_1$ leaves us with an action generated by $\bar A=(A+a) \times Id$ and $\bar B= (B+b) \times (\tilde{C}+\tilde{\p})$. The local rigidity of the action $\rho$ then requires a local rigidity result for the rank one action  $\tilde{C}+\tilde{\p}$ which obviously does not hold.}

{ For a $\Z^2$ partially hyperbolic affine action whose generators satisfy  (\ref{normalform}) it is possible to state  a rigidity theorem in a similar fashion as for 
 perturbations of quasi-periodic translations :}  Let $(f,g)$ be a perturbation of the generators $\bar A$ and $\bar B$ of such an action. First of all, since the linear parts of $f$ and $g$ are given by $A \times \text {Id}$,   $B \times \text {Id}$, on $\T^{d_1+d_2}$, one can define for any pair $\mu_1,\mu_2$ of invariant probability measures by $f$ and $g$ respectively  the translation vectors { along the $\T^{d_2}$ direction} corresponding to these measures as follows 
\begin{align*} \a&=\rho_{\mu_1}(f)= \int_{\T^{d}} \pi_2(f(x)-x) d\mu_1(x), \\
 \beta&=\rho_{\mu_2}(g)= \int_{\T^{d}} \pi_2(g(x)-x) d\mu_2(x)\end{align*}
where $\pi_2$ is the projection on the $\T^{d_2}$ variable.
We say that $(\a,\beta) \in \T^{d_2} \times \T^{d_2}$ is  simultaneously Diophantine with respect to a pair of numbers $(\lambda,\mu)$ if there exists $\tau,\gamma>0$ such that 
$$\max( |\lambda -e^{i2\pi (k,\a)}|, |\mu-e^{i2\pi (k,\beta)} |)> \frac{\gamma}{|k|^\tau}$$
where $\|\cdot\|$ denotes the closest distance to the integers, and we denote this property by $(\a,\beta)\in \text{SDC}(\tau,\gamma,\lambda,\mu)$.  We say that 
$(\a,\beta)\in \text{SDC}(\tau,\gamma,\bar A, \bar B)$ if given any pair of eigenvalues $(\lambda,\mu)$  of  $(\bar A,\bar B)$, it holds that $(\a,\beta)\in \text{SDC}(\tau,\gamma,\lambda,\mu)$.
Observe that SDC pairs of vectors relatively to any pair $(\bar A,\bar B)$ form a set of full Haar measure in   $\T^{d_2} \times \T^{d_2}$.

We have the following 
\begin{theo}\label{main2} Let $f,g$, be the generators of a smooth ($C^\infty$)  $\Z^2$ action on $\T^d$ such that  the linear part of $(f,g)$ is given by $\bar A=A \times \text {Id}_{\T^{d_2}}$,  $\bar B= B \times \text {Id}_{\T^{d_2}}$, with $A$ and $B$  two  
 commuting toral automorphisms of $\T^{d_1}$  that satisfy  (HR), $d_1+d_2=d$. For any $\tau,\gamma>0$, there exist $r(\tau)$ and $\eps(\tau,\gamma)$ such that if  for some pair of invariant probability measures $\mu_1,\mu_2$ by $f$ and $g$ respectively we have that $Ê(\a,\beta)=$ $(\rho_{\mu_1}(f),\rho_{\mu_2}(g)) \in  \text{SDC}(\tau,\gamma,\bar A, \bar B)$ and if $\|f-(A+a)\times T_\a\|_r\leq \eps$, $\|g-(B+b)\times T_\beta\|_r\leq \eps$,  { where $a, b\in \mathbb R^{d_1}$ and $T_\a$ and $T_\beta$ are translations of $\mathbb T^{d_2}$}; then the action is linearizable, namely there exists $h \in \text{Diff}^\infty (\T^d)$ such that $h \circ f \circ h^{-1} =  (A+a)\times T_\a$, $h \circ g \circ h^{-1} =  (B+b)\times T_\beta$.
\end{theo}

{ In the case $d_2=1$, the SDC condition is reminiscent of the one used by Moser to prove local rigidity of commuting circle diffeomorphisms with this condition on their rotation numbers \cite{M}. The ingredients of the proof of Theorem \ref{main2} are indeed a mixture of the ingredients used in the higher rank rigidity of toral automorphisms \cite{DK} and the KAM rigidity in the quasi-periodic setting as in \cite{A} and \cite{M}.  }

Also similar to the perturbations of quasi-periodic translations of the torus it is possible to state a rigidity theorem for a parametric family of $\Z^2$ actions.

Let $\rho_t$ be a family of $\Z^2$ actions where the parameter $t \in [0,1]$. 

Given $t$, the generators $ f_t, g_t$ of the $\Z^2$ action $\rho_t$  may be viewed as  $f_t=\bar A+\bar a_t$ and $g_t= \bar B+\bar b_t$, where  $\bar A$ and $\bar B$ generate  a linear action. If the linear action generated by $\bar A$ and $\bar B$ has a rank-one factor then up to a coordinate change in $\Z^2$ we may assume that the affine action $\rho_t$ is generated by 
$f_t=(A+a_t) \times (Id+\f(t))$ and $g_t= (B+b_t) \times (C+\p(t))$, where  $A$ and $B$ generate  a linear action, $(a_t,b_t)$ and $(\f(t), \psi(t))$ are translation vectors, and $C$ is a linear map. Arguing as in the case of a single action, we see that for any kind of rigidity to hold it is necessary that $C=Id$. Indeed, if $C$ is not Identity we can reduce to the case $f_t=A \times Id$ and $g_t=B \times (C+\p(t))$.
The latter can be perturbed into the family of actions generated by $f_t, \tilde{g}_t=B\times h_t$ with $h_t$ any perturbation of the family $C+\p(t)$ that can be chosen to be non linearizable for all $t$. 

To state a KAM rigidity result when $C= Id$ we need some transversality on the frequencies along the elliptic factor of the action. We will use a Pyartli  \cite{P} Êtype condition but other usual transversality conditions in KAM theory may be applied as well.

\begin{definition} We say that a function $\rho \in C^r([0,1], \T^d)$, $r\geq d$,  satisfies a Pyartli condition if for any $t \in [0,1]$ we have that the first $d$ derivatives of $\rho$ are linearly independent.  There exists then a constant $\nu>0$ such that 
\begin{equation} \label{pyartlinu}  |\text{det}(\rho', \rho'',\ldots, \rho^{(d)})| \geq \nu, \quad \| \rho \|_d \leq \nu^{-1} 
  \end{equation}
\end{definition}


\begin{theo}\label{main} Let $f_t,g_t$, $t\in [0,1]$ generate a family $\rho_t$ of  affine  $\Z^2$ actions on $\T^d$ which is of class $C^d$ in he  parameter  $t$. Then the following alternative holds in function of the common linear part $(\bar A,\bar B)$ of the family $(f_t,g_t)$.

(1) $(\bar A,\bar B)$ satisfies (HR) and every action in the family is locally rigid.

(2) $f_t=(A+a_t) \times (Id+\f(t))$ and $g_t= (B+b_t) \times (Id+\p(t)).$ 
 If  the function $n \f(t)+ m \p(t)$ satisfies a Pyartli condition for some $(n, m)\in \mathbb Z^2$  {  (for $d=d_2$ and with some constant $\nu$ and if in addition $\|\f\|_{d_2}, \| \psi \|_{d_2} \leq \nu^{-1}$ ),} then the family $\rho_t$ is KAM locally rigid: there exists { $r_0(A,B,n,m,d_2)$  such that  
for any $\eta$ there exists $\eps(\eta,n,m,\nu)$} such that if the action $\rho_t$ is perturbed into $\tilde{\rho}_t$ generated by $\tilde{f}_t$ and $\tilde{g}_t$ such that $\|\tilde{f}_\cdot-f_\cdot\|_{d,r_0}\leq \eps$,
$\|\tilde{g}_\cdot-g_\cdot\|_{d,r_0}\leq \eps$, then the set of parameters $t$ for which $(\tilde{f}_t,\tilde{g}_t)$  are simultaneously smoothly linearizable  is larger than $1-\eta$.  

(3) None of the actions in the family is locally rigid and the family is not KAM locally rigid : it can be perturbed so that no element of the perturbed family is linearizable.

\end{theo}

We denote by $\|\cdot \|_{d,r}$ the combination of $C^d$ norm in $t$ and $C^r$ norm in the torus variable. Part  (1) of  Theorem \ref{main} is proved in \cite{DK}. { Part (3) reduces as discussed above to the case $f_t=A \times Id$ and $g_t=B \times (C+\p(t))$.  As explained before,} in this paper we combine techniques from \cite{DK} and Arnol'd parameter exclusion technique for perturbations of quasi-periodic translations on the torus \cite{A}, to show Part (2) i.e., rigidity in the KAM sense for affine actions.

For the clarity of the exposition, the proof of Theorem \ref{main} will be first carried in detail only in the case $d_2=1$. The generalization to any $d_2$ is explained in Section \ref{generald2}.  Also, since the proof of Theorem \ref{main2} follows essentially the same lines as the proof of Theorem \ref{main}, we will only give a detailed proof of the former and explain in Section \ref{proofmain2} the main differences required for the proof of the latter.


Affine Anosov actions have been first 
discussed by Hurder in \cite{HurAffine}.  Local rigidity of hyperbolic and then partially hyperbolic affine actions of  higher rank non abelian groups was extensively studied (see for example the survey \cite{Fisher}).
In \cite{FM} Fisher and Margulis provide a complete local picture for affine actions by higher rank  lattices in semisimple Lie groups.  
The methods they use are totally different from ours and are 
speciÞc to groups with Property (T).

Prior to \cite{FM},  the question about local rigidity of perturbations of product actions of large higher rank groups has been addressed  in \cite{NT1}, \cite{NT2}, \cite{T}; the actions considered there are products of the identity action and actions that generalize the standard ${\rm SL}(n,\Z)$ action on $\T^n$.  Local rigidity and deformation rigidity are obtained for such actions. 
We  note that the actions we consider in this paper even though they belong to families of actions, are not deformation rigid in the sense of \cite{Hurder}. 

Local rigidity results for algebraic Anosov actions were obtained by Katok and Spatzier in  \cite{KS}, including the case of toral automorphisms and nilmanifold automorphisms. Currently not much is known about perturbations of affine actions on nilmanifolds when the linear part is a product of a higher rank abelian action and the identity, even when the higher rank abelian action is Anosov.

\subsection{Reduction to actions which are linear transversally to the elliptic factor} 



In the subsequent sections we give the proof of  Theorem \ref{main} in the case when the  { unperturbed} action  transversal to the elliptic factor is purely linear, namely  when $f_t=A \times R_{\f(t)}$ and $g_t= B\times R_{\p(t)}$, where $ R_{\f(t)}$ and $R_{\p(t)}$ denote translation maps on the circle. 
The same arguments  extend to the case when the { unperturbed} action transversal to the elliptic factor  is affine generated by  $A+a_t$ and $B+b_t$ instead of  $A$ and $B$ . The only difference is that in (2.9) the  number $\lambda_{m,t}$ should be replaced with $\lambda_{m,n,t}= e^{-i2\pi (m\varphi(t)+\langle n,a_t\rangle )}\lambda$. { This change  does not affect any subsequent estimates.}

\subsection{Exact statement of Theorem \ref{main} in the case of a one dimensional elliptic factor}

Let $A$ and $B$ be two commuting toral automorphisms satisfying (HR) condition.   
For $\f,\p \in {\rm Lip}(I_0,\R)$, $I_0=[0,1]$, let 
\begin{align*} f_{\f(t)}(x,\theta)&=(Ax,R_{\f(t)}(\theta)) \\ 
g_{\p(t)}(x,\theta)&=(Bx,R_{\p(t)}(\theta))  \end{align*}

For $I \in \R$, we denote $C^{lip,\infty}(I, \T^{d+1},\R^{d+1})$ the set of families of smooth maps  in the $\T^{d+1}$ variable and Lipschitz in the parameter $t \in  I$. We denote $C^{lip,\infty}_0(I,\T^{d+1},\R^{d+1})$ the subset of maps $f \in C^{lip,\infty}(I, \T^{d+1},\R^{d+1})$ such that if we write $f_t(z)=(f^1_t(z),f^2_t(z)) \in \T^d \times \T$, then $\int_{\T^{d+1}} f^2_{t}(z) dz =0$ for $t\in I$.

Consider 

\begin{align*} 
\tf_t(x,\theta)&=f_{\f(t)}(x, \theta)+\Delta f_t(x, \theta) \\
\tg_t(x,\theta)&=g_{\p(t)}(x, \theta)+\Delta g_t(x, \theta) 
\end{align*} 
with $\Delta f,\Delta g \in C^{lip,\infty}_0(I_0,\T^{d+1},\R^{d+1})$ and such that $\tf_t$ and $\tg_t$ commute for all $t\in I_0$. For $f \in C^{lip,\infty}_0(I_0,\T^{d+1},\R^{d+1})$, we use the notation  $\| f \|_{lip(I),r} = \max_{|\iota| \leq r} {\rm Lip}( f^{(\iota)})$ where ${\rm Lip}(f)$ is the maximum of the supnorm of $f$ and its Lipschitz constant, and $|\iota|$ is the maximal coordinate of the multi-index $\iota \in \N^{d+1}$.  We will also use the notation $\|v\|_{0(I),r}$ fort the supremum of the usual $C^r$ norms of $v(t)$ as $t\in I$.

{ Let $M$ be such that $$2\max(\|\f \|_{lip(I_0)},\|\p \|_{lip(I_0)}) \leq M, \quad \inf_{t\in I_0}\varphi'(t) \geq \frac{2}{M}$$}

\begin{theo} \label{main3} There exists $r_0(A,B) \in \N$  such that  for any $\eta$ there exists $\epsilon_0(A,B,M,\eta)>0$ such that if $\max(\| \Delta f \|_{lip(I),r_0}, \| \Delta g \|_{lip(I),r_0}) \leq \epsilon_0$, then the set of parameters $t$ for which the pair $\tf,\tg$ is simultaneously smoothly linearizable has measure larger than $1-\eta$.
\end{theo} 

Sections \ref{section.inductive} and \ref{sec.kam} below are devoted to the proof of Theorem \ref{main3}. Sections \ref{generald2} and \ref{proofmain2} explain how this proof should be modified to give the proof of Theorems \ref{main} and \ref{main2} respectivily. 

\bigskip 

{\bf Acknowledgments.}  The authors are grateful to Artur Avila, Hakan Eliasson, Anatole Katok and Rapha'l Krikorian for fruitful discussions and suggestions. 

\section{The inductive step} \label{section.inductive}

Let $\mathcal E(A)$ be the set of eigenvalues of $A$ union $1$.   For $N\in \mathbb N$, define 
$$\cD(N,A) = \{ \a \in I_0 \ / \ | \l - e^{i2\pi k\a}|\geq N^{-3}, \quad \forall \lambda \in \mathcal E(A), \forall 0<|k|\leq  N \}.$$

\begin{prop}\label{mainprop} There exists $\sigma(A,B)$ such that if $N \in \N$ and  $I$ is an interval such that $I \subset \{t \in I_0  \ / \ \varphi(t) \in \cD(N)\}$, then there  
 exist $\tilde{\f},\tilde{\p} \in {\rm Lip}(I,\R)$ and  $h,\widetilde{\Delta f}, \widetilde{\Delta g}  \in C^{lip,\infty}_0(I,\T^{d+1},\R^{d+1})$  such that if we write $ H=\id + h$ we have that 
\begin{equation}\label{nonlin}
\begin{aligned} H \circ \tf  &= (f_{\tilde{\f}} + \widetilde{\Delta f}) \circ H \\
 H \circ \tg &= (g_{\tilde{\p}} + \widetilde{\Delta g}) \circ H \end{aligned}
 \end{equation}
with   
\begin{align*}
\Delta S& \leq C_0 N^\sigma \Delta_0 \\
\|h\|_{lip(I),r+1}&\le  C_r S N^{\si}  \Delta_r +C_rSN^{\si}\Delta_0\Delta_r  \\
\td_r &\le  
C_r S N^{\si}\Delta_0\Delta_r +  C_{r,r'} N^{\sigma+ r-r'}\Delta_{r'}
\end{align*}
where:  

\begin{align*} 
S&=\max(\|\f \|_{lip(I)},\|\p \|_{lip(I)} ) \\
\Delta S&=\max(\| \f-\tilde{\f} \|_{lip(I)},\| \p-\tilde{\p} \|_{lip(I)}) \\
\Delta_r &=\max(  \|\Delta f\|_{lip(I),r},\|\Delta g\|_{lip(I),r} ) \\ 
\td_r &=\max(  \|\widetilde{\Delta f}\|_{lip(I),r},\|\widetilde{ \Delta g} \|_{lip(I),r} ) 
\end{align*}

\end{prop}

We will reduce the proof of Proposition \ref{mainprop} to the solution of a set of linear equations in the coordinates of $h$. These equations are solved using Fourier series and part of the solution is obtained with the higher rank techniques as in \cite{DK} while another part is obtained from solving linear equations above a circular rotation and requires  parameter exclusion to insure that the parameters that are kept satisfy adequate arithmetic conditions that allow to control the small divisors that appear.

\subsection{Reduction of the conjugacy step to linear equations}  \label{redconj}

By substituting $H=id+h$, the first equation in  \eqref{nonlin} becomes:
\begin{equation}\label{newerror}
\Delta f-(Df_{\tilde\f}h-h\circ f_\f)=f_{\tilde\f}-f_\f+\widetilde{\Delta f}(id+h)+E_{L, A}
\end{equation}
where $E_{L, A}=f_{\tilde\f}(Id+h)-f_{\tilde\f}-Df_{\tilde\f}h-h(f_\f+\Delta f)+h f_\f$
The map $Df_{\tilde\f}$ actually does not depend on $\tilde\f$, in fact it is the map $\bar A=(A, Id)$, where $A$ acts on $\mathbb R^d$ and $Id$ acts on $\mathbb R$. The second equation in  \eqref{nonlin} is linearized in the same way, so the linearization of \eqref{nonlin} is the system of  equations in $h$:
\begin{equation}\label{linconj}
\begin{aligned}
\bar Ah-h\circ f_\f&=\Delta f\\
\bar Bh-h\circ g_\p&=\Delta g
\end{aligned}
\end{equation}
where  $\bar B=(B, Id)$ and 
$E_{L, B}:=g_{\tilde\p}(Id+h)-g_{\tilde\p}-Dg_{\tilde\p}h-h(g_\p+\Delta g)+h g_\p$.

 Given a pair of commuting automorphisms $\bar A$ and $\bar B$ we call $(\lambda,\mu)$ a pair of eigenvalues of  $(\bar A,\bar B)$ if $\lambda$ and $\mu$ are eigenvalues of $\bar A$ and $\bar B$ for the same eigenvector. 

If $A$ and $B$ are semisimple, then by 
choosing a proper basis in $\R^d$ in which $A$ and $B$ simultaneously diagonalize, the system \eqref{linconj} breaks down into several systems of the following form  
\begin{equation}\label{simpleconj}
 \begin{aligned}
\la h-h\circ f_\f&=v\\
\mu h-h\circ g_\p&=w
\end{aligned}
\end{equation}
where   $\la$ and $\mu$ are a pair of eigenvalues of $A \times {\rm Id}$ and $B \times {\rm Id}$ and  $v, w \in C^{lip,\infty} (I\times \T^{d+1}, \mathbb R)$.  If $A$ and $B$ have non-trivial Jordan blocks, then instead of \eqref{simplecom}, for each Jordan block we would get a system of equations.  Lemma 4.4   in \cite{DK} shows that this system of equations  can be solved inductively in finitely many steps (the number of steps equals the size of a Jordan block), starting from equation of the form \eqref{simplecom}.  We will not repeat the argument here, instead we assume throughout that $A$ and $B$ are semisimple and we refer to  Lemma 4.4   in \cite{DK} for the general case.

\subsection{Reduction of the  commutativity relation}

Since $f_\f$ and $g_\p$ commute and are linear, the equation  $(f_\f+\Delta f)\circ (g_\p+\Delta g)=(g_\p+\Delta g)\circ (f_\f+\Delta f)$ reduces to:
$$\bar A\Delta g-\Delta g(f_\f+\Delta f)=\bar B\Delta f-\Delta f(g_\p+\Delta g)$$
If we push the terms linear in $\Delta f$ and $\Delta g$ to the left and all the non-linear terms to the right hand side we obtain
\begin{equation}\label{lincom}
\bar A\Delta g-\Delta g\circ f_\f-\bar B\Delta f-\Delta f\circ g_\p=\Phi
\end{equation}
where
\begin{equation}\label{Phi}
\Phi=\Delta g(f_\f+\Delta f)-\Delta g\circ f_\f-(\Delta f(g_\p+\Delta g)-\Delta f\circ g_\p).
\end{equation}


Similarily to section \ref{redconj}, if $A$ and $B$ are semisimple, the equation  \eqref{lincom} reduce to several equations  of the  form:
 \begin{align}\label{simplecom}
(\la w-w\circ f_\f)-(\mu v-v\circ g_\p)=\phi
\end{align} 


\subsection{An approximate solution to \eqref{simpleconj}}


The main result in this note is that the system of linear equations \ref{simpleconj} can be solved up to an error term that is controlled by $\Phi$ which is  quadratically small  in the perturbation terms $\Delta f, \Delta g$.  

\begin{lemm}\label{splittinglemma} For $v, w, \phi \in C^{lip, \infty}(I\times \T^{d+1}, \R)$ satisfying \eqref{simplecom}, and $\la\ne 1, \mu\ne 1$, if $N \in \N$ and  $I$ is an interval such that $I \subset \{t \in I_0  \ / \ \varphi(t) \in \cD(N)\}$, then there  exists $h \in C^{lip, \infty}(I\times \T^{d+1}, \R)$ such that: 
\begin{equation*}\label{split}
\begin{aligned}
&\|h\|_{lip(I), r+1}\le  C_rSN^{\si}\|v\| _{lip(I),r}+C_rSN^{\si}\|\phi\|_{lip(I), r-2}\\
&\|v-(\la h-h\circ f_\f)\|_{lip(I), r}\le  C_{r,r'}N^{d+ r-r'}\|v\|_{lip(I), r'}+C_rSN^{\si}\|\phi\|_{lip(I), r-2}\\
&\|w-(\mu h-h\circ g_\p)\|_{lip(I), r}\le  C_{r, r'}N^{d+r'-r}\|w\|_{lip(I), r'}+ C_rSN^{\si}\|\phi\|_{lip(I), r-2}
\end{aligned}
\end{equation*}
for all $r'>r\ge 0$ and $\si=\si(A, B, \la, \mu, d)$.
The same holds true for $(\la, \mu)=(1,1)$ provided $v, w\in C_0^{lip, \infty}(I\times \T^{d+1}, \R)$.
 \end{lemm} 
 

 \subsection{Proof of Proposition \ref{mainprop}}
 Before we give the proof of Lemma \ref{splittinglemma}, we show how it implies Proposition \ref{mainprop}.

By applying Proposition \ref{annexe.compose} from the Appendix to the equation \eqref{Phi}, we have that 
\begin{equation} \|\Phi\|_{lip(I), r-2}\le C_r \Delta_0\Delta_r \label{quad}
 \end{equation}


Since $\Delta f,\Delta g \in C_0^{lip, \infty}(I\times \T^{d+1}, \R^{d+1})$ we can apply Lemma \ref{splittinglemma} to all the coordinates in (\ref{linconj}) and get $h$ such that 
\begin{equation*}\label{finalHest}
\begin{aligned}
&\|h\|_{lip(I), r+1}\le  C_rSN^{\si}\Delta_r+C_rSN^{\si}\Delta_0\Delta_r \\
&\|\Delta f-(\bar A h-h\circ f_\f)\|_{lip(I), r}\le C_rSN^{\si}\Delta_0\Delta_r+  C_{r,r'}N^{d+ r-r'}\Delta_{r'}\\
&\|\Delta g-(\bar B h-h\circ g_\p)\|_{lip(I), r}\le  C_rSN^{\si}\Delta_0\Delta_r+ C_{r, r'}N^{d+r'-r}\Delta_{r'}
\end{aligned}
\end{equation*}
where the new constant $\si$ is $d$ times the constant $\sigma$ from Lemma \ref{splittinglemma}. For the bound on $h$ we use Lemma  \ref{splittinglemma} and \eqref{quad} with $r'=r$.
In light of \eqref{newerror}, we take
\begin{equation}\label{tildes}
\begin{aligned}
\tilde\f&:= \f +Ave(E_{L,A}^2 \circ (\text{Id} + h)^{-1}) \\
\tilde\p&:= \p+Ave(E_{L,B}^2 \circ (\text{Id} + h)^{-1}) 
\end{aligned}
\end{equation}
and let 
$$\widetilde{\Delta f}=\left(( \Delta f_0-(A h-h\circ f_\f)) -E_{L, A})\right) \circ  (\text{Id} + h)^{-1} +f_\f-f_{\tilde \f}  $$
with a similar definition for $\widetilde{\Delta g}$.
 
\noindent {\it Claim.} We have that $\tilde \f, \tilde \p$, $h$ and $\widetilde{\Delta f}, \widetilde{\Delta g}$ satisfy the conclusion of Proposition \ref{mainprop}.
 

\carre

The rest of Section \ref{section.inductive} is devoted to the proof of Lemma  \ref{splittinglemma}.

 \subsection{Proof of Lemma \ref{splittinglemma}}
 
We first describe obstructions for solving a single coboundary equation in \eqref{simpleconj}. For a fixed $t\in I$ the first equation in \eqref{simpleconj} becomes:
\begin{equation}\label{coboundary}
\la h_t-h_t\circ f_{\f(t)}=v_t,
\end{equation}
where $h_t=h(t, \cdot)$ and similarily for $v$ and $w$.
By reducing to Fourier coefficients, for every $(n, m)\in \ZZ^d\times \ZZ$ we have:
\begin{equation*}
\la \sum_{(n, m)} h_{n, m, t}\chi_{n,m}(x,\theta)-\sum h_{n, m,t}\chi_{n,m}(Ax,\theta+\vp(t))=\sum v_{n, m, t} \chi_{n,m}(x,\theta)
\end{equation*}
\begin{equation*}
 \sum_{(n, m)} (\la h_{n, m, t}-h_{A^*n, m, t}e^{2\pi i m\vp(t)}\chi_{n,m}(x,\theta))=\sum v_{n, m,t} \chi_{n,m}(x,\theta),
\end{equation*}
where $h_{n, m, t}$ denotes  the $(n, m)$-Fourier coefficient of the function $h_t$, $\chi_{n,m}(x,\theta)=e^{2\pi i (n\cdot x+m\t)}$, and $A^*=(A^t)^{-1}$.
Thus for every $(n, m)$
\begin{equation*}
\la h_{n, m,t}- h_{A^*n, m,t}e^{2\pi i m\vp(t)}= v_{n, m,t}.
\end{equation*}
By denoting: $\la_{m,t}:=e^{-2\pi i m\vp(t)}\la$ and $ v'_{n, m,t}:=e^{-2\pi i m\vp(t)}v_{n, m,t}$, we have
\begin{equation}\label{oneeq}
\la_{m, t} h_{n, m,t}- h_{A^*n, m,t}= v'_{n, m,t}.
\end{equation}

For a fixed $m$ and $n\ne 0$ and for a fixed $t$ the equation \eqref{oneeq} has been discussed in \cite{DK}; the obstructions are precisely defined as well as the construction which allows for removal of all the obstructions (Lemma 4.5 in \cite{DK}). The obstructions are:

\begin{equation}\label{obs}
O_{n, m}^A(v_t)=\sum_{k\in \ZZ} \la_{m,t}^{-(k+1)} v'_{A^kn, m,t},
\end{equation}
where we abuse notation a bit by using $A^k$ to denote the $k$-th iterate of the dual map $A^*$. 
The proof of  Lemma \ref{splittinglemma} relies on two claims. In the first one we solve a system of the type  \eqref{simpleconj} provided a set of obstructions computed with the right hand side vanish. In the second claim, we show how the commutation relation  allows to modify the right hand side in  \eqref{simpleconj}  to set the obtructions to $0$. Moreover, the modification is of the order of the "commutation error" $\Phi$ in \eqref{lincom}.

\bigskip

\noindent {\bf Claim 1.} {\it Let $v$ be in $C^{lip(I), \infty}(I,\T^{d+1}, \RR)$ such that for all $t\in I$ and $|m|>N$, $v_{0, m,t}=0$. If for all $n, m,$ and  $t\in I$, $n\neq 0$, $O_{n, m}^A(v_t)=0$, and $ave(v_t)=0$ in the case $\la=1$, then there exists a  solution $h$ to the equation $\la h-h\circ f_\f =v$ in $C^{lip(I), \infty}(I\times\T^{d+1}, \RR)$ satisfying  
\begin{equation}\label{h-est}
\|h\|_{lip(I), r}\le C_rSN^3\|v\|_{lip(I), r+\si}
\end{equation}
for all $r\ge 0$, where $\sigma=\sigma\{\lambda, d, A\}$. Moreover, if $h$ and $v$ are smooth maps with $h_{0, m,t}=v_{0, m,t}=0$  for $|m|>N$ and with averages zero, such that $\la h-h\circ f_\f =v$ on $I\times \T^{d+1}$, then $h$ satisfies the estimate \eqref{h-est}.}

\bigskip

{\it Proof of the Claim 1}. The proof is similar to the proof of  the Lemma 4.2 in \cite{DK}, except that here one extra (isometric) direction causes somewhat greater loss of regularity. 

Solution $h$ is defined via its Fourier coefficients $h_{n, m,t}$, each of which can be defined, in case $n\ne 0$, by using one of the two forms: 


\begin{equation}\label{h-def}
h_{n, m,t}=\sum_{k=0}^\infty \la_{m,t}^{-(k+1)} v'_{A^kn, m,t}=-\sum_{k=-\infty}^{-1} \la_{m,t}^{-(k+1)} v'_{A^kn, m,t}.
\end{equation}

One can use one or the other form to obtain an estimate for the size of $h_{n, m,t}$ depending on whether a non-trivial $n$ is largest in the expanding or in the contracting direction for $A$. This is completely the same as in \cite{DK} and it automatically gives an estimate of the size of $h_{n, m,t}$ with respect to the norm of $n$. In order to obtain here the full estimate for the $C^r$ norm of $h$ we need to estimate the size of $h_{n, m,t}$ with respect to the norm of $(n, m)$ and this is the only additional detail needed here. But this is not a big problem: since $n$ is non-trivial, after approximately $\log m$ iterations of $n$ by $A$, the resulting vector will surely be larger than $m$. 
We have:

\begin{equation}\label{h-est-1}
\begin{aligned}
|h_{n, m,t}|&\le \sum_{k=0}^{\infty}|\la_{m,t}^{-(k+1)}||v'_{A^kn, m,t}|\\&=\sum_{k=0}^{\infty}|\la|^{-(k+1)}|v_{A^kn, m,t}|\\&\le \|v\|_{0(I),r}\sum_{k=0}^{\infty}|\la|^{-(k+1)}\|(A^kn, m)\|^{-r}.
\end{aligned}
\end{equation}
Take the norm in $\ZZ^N\times \ZZ$ to be $\|(n, m)\|=\max\{\|n\|, |m|\}$, where for $n\in \ZZ^N$, $\|n\|$ is the maximum of euclidean norms of projections of $n$ to expanding, contracting and the neutral directions for $A$. Let $n_{exp}$ denote the projection of $n$ to the expanding subspace for $A$. Due to ergodicity of $A$ this projection is non-trivial. For example we say that $n$ is largest in the expanding if $\|n_{exp}\|\ge C\|n\|$ where $C$ is a fixed constant ($C=1/3$ works). Similarily, we say that $n$ is largest in the contracting (resp. neutral) direction if the projection of $n$ to the contracting (resp. neutral) direction is greater than constant times the norm of $n$. 


Then if $\rho$ denotes the expansion rate for $A$ in the expanding direction for $A$,  we have by the Katznelson Lemma (See for example Lemma 4.1 in \cite{DK}): 
\begin{equation*}
\begin{aligned}
\|(A^kn, m)\|&\ge \max\{\|A^kn_{exp}\|, |m|\}\ge \max\{\rho^k\|n_{exp}\|, |m|\}\\
&\ge \max\{C\rho^k\|n\|^{-d}, |m|\}\ge \max\{C\rho^{k-k_0}\rho^{k_0}\|n\|^{-d}, |m|\} 
\end{aligned}
\end{equation*}

 Since $\rho^{k}\|n\|^{-d}\ge \|(n, m)\|$ for all $k\ge \frac{d+1}{\ln \rho}\ln \|(n, m)\|$, we have: 
$$ \|(A^kn, m)\|\ge C\rho^{k-k_0}\|(n, m)\|$$
for all $k> k_0=[\frac{d+1}{\ln \rho}\ln \|(n, m)\|]$.

Now if $n$ is largest in the expanding direction for $A$ then for $0\le k\le k_0$:   $\|(A^kn, m)\|\ge C\|(n, m)\|$. If $n$ is largest in the neutral direction for $A$, then for $0\le k\le k_0$:   $\|(A^kn, m)\|\ge C(1+k)^{-d}\|(n, m)\|$. 
 
Thus for all $t\in I$ (in the worst case scenario, when $|\la|<1$):

\begin{equation*}\label{h-est-11}
\begin{aligned}
|h_{n, m,t}|&\le \|v\|_{0(I),r}(\sum_{k=0}^{k_0}|\la|^{-(k+1)}\|(n, m)\|^{-r}+\sum_{k=k_0}^{\infty}|\la|^{-(k+1)}(C\rho^k\|(n, m)\|)^{-r})\\
&\le \|v\|_{0(I),r}(k_0|\la|^{-(k_0+1)}\|(n, m)\|^{-r}+|\la|^{k_0}(C\rho^{k_0}\|(n, m)\|)^{-r}\sum_{k=0}^\infty |\la|^{-k}\rho^{-kr}
\\
&\le C_r\|v\|_{0(I),r}(\|(n, m)\|^{\frac{d+1}{ln\rho}}\log \|(n, m)\|)\|(n, m)\|^{-r}+(C\rho^{k_0}\|(n, m)\|)^{-r}) \\
&\le C_r\|v\|_{0(I),r}\|(n, m)\|^{-r+\si} 
\end{aligned}
\end{equation*}
where $\si= 2+d+ a+\delta$, $\delta>0$, and $a=a(\la)=\frac{d+1}{ln\rho}>0$ in general depends only on the eigenvalues of $A$. 
Note that for the convergence of the sum $\sum_{k=0}^\infty |\la|^{-k}\rho^{-kr}$ it suffices to assume that the regularity $r$ of $v$ is greater than a constant $-\frac{\ln|\la|}{\ln\rho}$, which in general depends on eigenvalues of $A$. We recall that the norm $\|v\|_{0(I),r}$ denotes the supremum of the usual $C^r$ norms of $v(t)$ as $t\in I$.


When $n$ is largest in the contracting direction for $A$ then just as in \cite{DK} we repeat the above estimates using the  expression $h_{n, m,t}=-\sum_{k=-\infty}^{-1} \la_{m,t}^{-(k+1)} v'_{A^kn, m,t}
$ for the coefficients $h_{n, m,t}$ instead to obtain the same bound: $|h_{n, m,t}|\le  C_r\|v\|_{0(I),r}\|(n, m)\|^{-r+\si}$, where $\si$ is now slightly different (changed by a constant) to include the eigenvalues for $A$ in the contracting directions. 


Now in the case $n=0$, and any non-zero $m$ the equation \eqref{oneeq} implies:

\begin{equation*}
\la h_{0, m,t}-h_{0, m,t}e^{2\pi i m\vp(t)}=v_{0, m,t},
\end{equation*}
so in this case 
\begin{equation}\label{h-0}
h_{0,m,t}=\frac{v_{0, m,t}}{\la-e^{2\pi i m\vp(t)}},
\end{equation}
and thus for $|\la|\ne 1$ we have that for all $t\in I$:
$$|h_{0,m,t}|\le (|\la|-1)^{-1} \|v\|_{0(I),r}|m|^{-r}.$$
In the case $|\la|=1$ 
 this is a small divisor problem. 
 When $t\in I$ we have $\varphi(t)\in \cD(N)$ and thus for $|m|\le N$ we have:
$$|h_{0,m,t}|\le N^3 |v_{0,m,t}|\le \|v\|_{0(I),r}N^3|m|^{-r}$$
Since for $|m|>N$, $v_{0,m}=0$, we define $h_{0, m}=0$ for $|m|>N$.

Accumulating all the estimates, we have for all $t\in I$:
$$|h_{n,m,t}|\le C_r \|v\|_{0(I),r}N^3\|(n, m)\|^{-r+\si}.$$

Thus the function $h$ defined via its Fourier coefficients $h_{n, m,t}$ satisfies the equation $\la h-h\circ f_\p =v$ and the estimate:
\begin{equation}\label{h-total-0}
\|h\|_{0(I),r}\le C_rN^{3}\|v\|_{0(I),r+\si},
\end{equation}
for all $r>r_0$, where $\si$ is a fixed constant, $\si=d+2+\max\{|\la|, |\la|^{-1}\}$, which in our set-up  depends only on the eigenvalues of $A$ and the dimension of the torus.

We estimate now $h$ in the direction of the parameter $t$. First we can characterize $x\in C^{lip, \infty}(I, \T^{d+1},\RR)$ by a property of Fourier coefficients of $x$. Let $\Delta x:=x_t-x_{t'}$, and similarly $\Delta x_{n, m}=x_{n, m, t}-x_{n, m, t'}$. Namely, $x\in C^{lip, s}(I, \T^{d+1},\RR)$ implies not only that that $x_{n, m, t}$ decay faster than $\|(n, m)\|^{-s}$ but also  from
$\|\Delta x^{(s)}\|_0\le L_s|t-t'|$ we get that $|\Delta x_{n, m}|\le C_s\|(n, m)\|^{-s}|t-t'|$ for some constant $C_s$. It is then easy to check that $|\Delta x_{n, m}|\le C_s\|(n, m)\|^{-s-d-1} |t-t'|$ suffices for  $x\in C^{lip, s}(I, \T^{d+1},\RR)$.

By using \eqref{h-def} (denote for simplicity by $\Sigma^{\pm}$ positive or negative sum in \eqref{h-def}) we have for $n\ne 0$:
\begin{equation*}
\begin{aligned}
&|\Delta h_{n, m}|=|\Sigma^{\pm}\la^{-(k+1)}(e^{2\pi i km \vp(t)}v_{A^kn, m, t}-e^{2\pi i km \vp(t')}v_{A^kn, m, t'})|\\
&=|\Sigma^{\pm}\la^{-(k+1)}((e^{2\pi i km \vp(t)}-e^{2\pi i km \vp(t')})v_{A^kn, m, t} +e^{2\pi i km \vp(t')}\Delta v_{A^kn, m, t})|\\
&\le (2\pi \|\f\|_{lip(I)}\|v\|_{0(I), r}+\|v\|_{lip(I), r}) |t-t'|\Sigma ^{\pm}|\la|^{-(k+1)}|k|\|(A^kn, m)\|^{-r+1} 
\end{aligned}
\end{equation*}
and from the discussion following \eqref{h-est-1} we have that for every $(n, m)$, $n\ne 0$, either the positive or the negative sum in the last expression above can be bounded by $C_r\|n, m\|^{-r+\sigma+1}$. When $n=0$ from \eqref{h-0} and for $t, t'\in I$ it is clear that $\Delta h_{0, m}\le CN^3\Delta v_{0, m}$. This gives the bound for the  Lipschitz constant for any $r$-th  derivative of $h$ 
which combined with \eqref{h-total-0} implies $\|h\|_{lip(I), r-\si-2-d}\le C_rN^3S\|h\|_{lip(I), r}$.

For the second part of the claim, if $h$ and $v$ are smooth and satisfy $\la h-h\circ f_\f =v$ for $t\in \mathcal I$ then for $n\ne 0$ the obstructions $O_{n, m}^A(v_t)$ are all zero, thus if $v$ satisfies in addition that $v_{0, m,t}=0$ for $|m|>N$ then by the first part of the Claim 1 there exists $h'$ such that $\la h'-h'\circ f_\f =v$ on $I$ and satisfies the estimate \eqref{h-est}. Then for $h"=h-h'$,  $\la h"=h"\circ f_\f$ on $I $, but this implies $h"=0$ in case $\la\ne 1$, or is constant in case $\la=1$. However, by construction $h'$ has average 0, and so does $h$ by assumption, so in any case $h=h'$ on $I$, which implies that $h$ satisfies the estimate \eqref{h-est}. \\
{\it -End of pf of claim 1.-}

Remark. The following fact will be used in the proof of the Claim 2 bellow: For every $n\in \ZZ^d$ there exists a point $n^*$ on the orbit $\{A^kn\}_{k\in\ZZ}$ such that the projection of $n$ to the contracting subspace of  $A$ is larger than the projection to the expanding subspace of $A$ and for $An$ the opposite holds:  projection of $An$ to the contracting subspace of  $A$ is smaller than the projection to the expanding subspace of $A$. For each $n$ choose an $n^*$ on the orbit of $n$ with this property \cite{DK}.

\bigskip

{\bf Claim 2.} {\it  Assume that for all $t\in I$ the following holds:
\begin{equation}\label{simplecom-t}
(\la w_t-w_t\circ f_{\f(t)})-(\mu v_t-v_t\circ g_{\p(t)})=\phi_t
\end{equation}
 and $v_{0, m,t}=w_{0, m,t}=\phi_{0,m,t}=0$ for $|m|>N$. 
Define $\tilde v_t$ by

\begin{equation*}
\tilde v_{n, m,t} = \left\{ \begin{aligned}
   O^A_{n, m}(v_t), \,\, & n\ne 0, n=n^* \\
    0, \,\, & \mbox{otherwise.} \end{aligned}\right. 
    \end{equation*}
Then: 

(1) For $n\ne 0$,  $O_{n, m}^A(v_t-\tilde v_t)=0$.

(2) $\|\tilde v\|_{lip(I), r}\le C_r N^3S\|\phi\|_{lip(I), r+\si}$, where $\si=\si(A, B, \la, \mu, d)$ and $r\ge 0$.
}


\bigskip

{\it Proof of claim 2.}

(1) This is immediate from the definition of $O_{n, m}^A$ and $\tilde v_t$.

(2) In Fourier coefficients \eqref{simplecom-t} becomes:

\begin{equation*}\label{fcoefM}
(\la w_{n, m,t}- w_{An, m,t}e^{2\pi i m\vp(t)})-(\mu v_{n, m,t}- v_{Bn, m,t}e^{2\pi i m\psi(t)})=\phi_{n, m,t}
\end{equation*}
which implies that for non-zero $n$ the obstructions $O_{n, m}^A$ for $$(\mu v_{n, m,t}- v_{Bn, m,t}e^{2\pi i m\psi(t)})+\phi_{n, m,t}$$ are trivial. This implies that $O_{n, m}^A(v_t)$ satisfies the equation:
\begin{equation}\label{obs-eq}
\mu O_{n, m}^A(v_t)- e^{2\pi i m\psi(t)} O_{Bn, m}^A(v_t)= O_{n, m}^A(\phi_t)
\end{equation}
where $ O_{n, m}^A(v_t)$ and $O_{n, m}^A(\phi_t)$ are defined as in \eqref{obs}.
From this, by backward and forward iteration by $B$, one obtains two expressions for $ O_{n, m}^A(v_t)$:
\begin{equation*}\begin{aligned}
 O_{n, m}^A(v_t)&=\sum_{l\ge 0} \mu_{m,t}^{-(l+1)} e^{-2\pi i m\psi(t)}O_{B^ln, m}^A(\phi_t)\\
 &=-\sum_{l< 0} \mu_{m,t}^{-(l+1)} e^{-2\pi i m\psi(t)}O_{B^ln, m}^A(\phi_t),\end{aligned}
 \end{equation*}
where $ \mu_{m,t}:=e^{-2\pi i m\psi(t)}\mu$.

It is proved in Lemma 4.5 in \cite{DK} that if every $A^kB^l$ for $(k, l)\ne (0, 0)$ is ergodic, and if $n=n^*$ then either for $l>0$ for $l<0$, the term  $\|(B^lA^kn, m)\|$ has exponential growth in $(l, k)$ for $\|(l, k)\|$ larger than some $C\log|n|$ and polynomial growth for $\|(l, k)\|$ less than $C\log|n|$. 
Hence, for $n=n^*$, it follows exactly as in Lemma 4.5  \cite{DK},  that either one or the other sum above are comparable to the size of $\|\phi_{t}\|_r\|(n, m)\|^{-r+\sigma}$, where $\si$ is a constant which depends only on $A, B $ and the dimension $d$.
Therefore in case $n\ne 0$ for all $t\in I$
\begin{equation}\label{tildev}
|\tilde v_{n, m,t}|=|O_{n, m}^A(v_t)|\le C_r\|\phi\|_{0(I),r}\|(n, m)\|^{-r+\sigma}.
\end{equation}
This implies the $\|\cdot\|_{0(I), r}$-norm estimate for $\tilde v$.

To obtain the estimate in the $t$ direction, just as in the Claim 1,  we look at $\Delta \tilde v_{n, m}$.  For $n\ne 0, n=n^*$:
\begin{equation*}
\begin{aligned}
&\Delta \tilde v_{n, m}=O_{n, m}^A(v_t-v_t')=\\
&\Sigma_l^{\pm}\Sigma_k \mu^{-(l+1)}\la^{-(k+1)}e^{-2\pi i m (l\psi(t)+k\vp(t))}(\phi_{B^lA^kn, m, t}-\phi_{B^lA^kn, m, t'})\\
&+\Sigma_l^{\pm}\Sigma_k \mu^{-(l+1)}\la^{-(k+1)}(e^{-2\pi i m (l\psi(t)+k\vp(t))}-e^{-2\pi i m (l\psi(t')+k\vp(t'))})\phi_{B^lA^kn, m, t'}.\\
\end{aligned}
\end{equation*}
If $\vp$ and $\psi$ are Lipschitz and $\phi$ is in $C^{lip(I), r}$, we have:
\begin{equation*}
\begin{aligned}
&|\Delta \tilde v_{n, m}|\le 
\|\phi\|_{lip(I), r}|t-t'|\Sigma_l^{\pm}\Sigma_k |\mu|^{-(l+1)}|\la|^{-(k+1)}\|(B^lA^kn, m)\|^{-r}\\
&+2\pi S|t-t'|\|\phi\|_{0(I), r}\Sigma_l^{\pm}\Sigma_k \mu^{-(l+1)}\la^{-(k+1)}|k||l|\|(B^lA^kn, m)\|^{-r+1}\\
\end{aligned}
\end{equation*}
Now the same argument as above (based on  Lemma 4.5  \cite{DK}) implies that for every $n=n^*$ one of the sums (for $l>0$ or $l<0$) \\
$\Sigma_l^{\pm}\Sigma_k \mu^{-(l+1)}\la^{-(k+1)}|k||l|\|(B^lA^kn, m)\|^{-r+1}$ can be bounded by $\|(n, m)\|^{-r+\sigma}$, where $\sigma$ is a constant depending on $A, B, \la, \mu$ and $d$. This implies 
$$|\Delta \tilde v_{n, m}|\le CS\|\phi\|_{lip(I), r}\|(n, m)\|^{-r+\sigma}|t-t'|.$$
Taking into account all the estimates above, this implies:
$$\| \tilde v\|_{lip(I), r}\le C_rN^3S\|\phi\|_{lip(I), r+\si},$$
with $\si$ fixed depending only on $A, B, \la$ and $d$.

{\it -End of proof of claim 2-}

\vspace{0.2in}

Now given $v, w$ such that $(\la w-w\circ f_{\f})-(\mu v-v\circ g_{\p})=\phi$, first truncate $v_t$ to $T_Nv_t$ for all $t\in I$. We choose the same $N$ for all $t\in I$. The truncation and the residue satisfy the following estimates for every $t$ and $r\leq r'$ 
\begin{equation}\label{truncest}
\begin{aligned}
\|T_Nv_t\|_{r'}&\le C_{r,r'} N^{r'-r+d}\|v_t\|_{r}\\
\|R_Nv_t\|_{r}&\le C_{r,r'}N^{r-r'+d}\|v_t\|_{r'}
\end{aligned}
\end{equation}
Since the same truncation is used for all t, it is easy to check that

\begin{equation*}
\begin{aligned}
\|T_Nv\|_{lip(I),r'}&\le C_{r,r'}N^{r'-r+d}\|v\|_{lip(I),r}\\
\|R_Nv\|_{lip(I),r}&\le C_{r,r'}N^{r-r'+d}\|v\|_{lip(I),r'} 
\end{aligned}
\end{equation*}
Now the Claim 2 applies to $T_Nv$. It gives $\widetilde {T_Nv}$ such that for $T_Nv-\widetilde {T_Nv}$ the obstructions $O_{n, m}^A(T_Nv_t-\widetilde {T_Nv_t})$ vanish for $n\ne 0$ and $$\|\widetilde {T_Nv}\|_{lip(I),r}\le C_rN^3S\|T_N\phi\|_{lip(I),r+\si}.$$ 
Notice that $\widetilde{T_Nv_t}$ by construction has all $(0, m,t)$-Fourier coefficients equal to zero for $|m|>N$.
Thus the Claim 1 can be applied to $T_Nv-\widetilde {T_Nv}$. Therefore there exists $h\in \ci(\cA\times \T^{d+1}, \R^{d+1})$ as in Claim 1 such that for all $t\in \cA$: $$T_Nv_t-\widetilde{T_N v_t}=\la h_t-h_t\circ\ft$$ and 

\begin{equation}\label{he}
\begin{aligned}
\|h\|_{lip(I),r+1}&\le C_rN^3S\|T_Nv-\widetilde{T_N v}\|_{lip(I),r+1+\si}\\
&\le C_rN^3S(\|T_Nv\| _{lip(I),r+1+\si}+ C_rN^3\|T_N\phi\|_{lip(I),r+1+2\si})\\
&\le  C_rSN^{4+\si}\|v\| _{lip(I),r}+C_rSN^{6+2\si}\|\phi\|_{lip(I),r-2}.
\end{aligned}
\end{equation}
Also 
\begin{equation*}
\begin{aligned}
\|v-(\la h-h\circ f)\|_{lip(I),r}&= \|R_Nv+\widetilde{T_N v}\|_{lip(I),r}\\
&\le \|R_Nv\|_{lip(I),r}+ C_rSN^3\|T_N\phi\|_{lip(I),r+\si}\\
&\le C_{r,r'}N^{r-r'+d}\|v\|_{lip(I),r'}+C_rSN^{5+\si}\|\phi\|_{lip(I),r-2}
\end{aligned}
\end{equation*}
Now to estimate $w-(\mu h-h\circ g)$ we use:

\begin{equation*}
\begin{aligned}
&(\la w-w\circ f)-(\mu v-v\circ g)=\phi\\
&(\la w-w\circ f)-(\mu T_Nv-T_Nv\circ g)-(\mu R_Nv-R_Nv\circ g)=\phi\\
&(\la w-w\circ f)-(\mu (T_Nv-\widetilde {T_Nv})-(T_Nv-\widetilde{T_N v}) \circ g)\\
&-(\mu \widetilde{T_N v}-\widetilde{T_N v}\circ g)-(\mu R_Nv-R_Nv\circ g)=\phi\\
&(\la w-w\circ f)-(\mu (\la h-h\circ f)-(\la h-h\circ f) \circ g)-(\mu \widetilde{T_N v}-\widetilde{T_N v}\circ g)\\
&-(\mu R_Nv-R_Nv\circ g)=\phi\\
&\la (w-(\mu h-h\circ g))-(w-(\mu h-h\circ g))\circ f=\\
&\phi+(\mu \widetilde{T_N v}-\widetilde{T_N v}\circ g)-(\mu R_Nv-R_Nv\circ g).\\
\end{aligned}
\end{equation*}
This implies:
\begin{equation*}
\begin{aligned}
&\la (T_Nw-(\mu h-h\circ g))-(T_Nw-(\mu h-h\circ g))\circ f=\\
&\phi+(\mu \widetilde{T_Nv}-\widetilde{T_N v}\circ g)-(\mu R_Nv-R_Nv\circ g)-(\mu R_Nw-R_Nw\circ g)=\\
&T_N\phi +(\mu \widetilde{T_Nv}-\widetilde{T_N v}\circ g).
\end{aligned}
\end{equation*}
Since both $T_Nw-(\mu h-h\circ g)$ (by construction of $h$) and 
$T_N\phi +(\mu \widetilde{T_Nv}-\widetilde{T_N v}\circ g)$
(by construction of $\widetilde{T_Nv}$), satisfy that their $(0, m,t)$ Fourier coefficients are zero for $|m|>N$,  the second part of the Claim 1 applies and gives an estimate for $T_Nw-(\mu h-h\circ g)$:
\begin{equation*}
\begin{aligned}
\|T_Nw-(\mu h-h\circ g)\|_{lip(I),r}&\le C_rSN^3\|T_N\phi+(\mu  \widetilde{T_Nv}- \widetilde{T_Nv}\circ g)\|_{lip(I),r+\si}\\
&\le C_rSN^{5+2\si}\|\phi\|_{lip(I),r-2}\\
\end{aligned}
\end{equation*}
Therefore:
\begin{equation*}
\begin{aligned}
\|w-(\mu h-h\circ g)\|_{lip(I), r}&\le C_rSN^{5+2\si} \|\phi\|_{lip(I),r-2}+\|R_Nw\|_{lip(I), r}\\
&\le  C_rSN^{5+2\si}\|\phi\|_{lip(I),r-2}+C_{r, r'}N^{d+r'-r}\|w\|_{r'}
\end{aligned}
\end{equation*}
Finally we can  redefine  the constant $\sigma$ by $\sigma:= 6+2\si$. This completes the estimates in Lemma \ref{splittinglemma}.


\section{The KAM scheme} \label{sec.kam}

\begin{lemm} \label{exclusion} Let $M>0$. There exists $N_0(M)$ such that if $N>N_0$ and $\tilde{N}=N^{3/2}$ and  if $I$ is an interval of size $1\geq |I|\geq 1/(2MN^2)$ and if $M^{-1}<\varphi'(t)<M$ for every $t \in I$, then there exists a union of disjoint intervals $\cU = \{ \tI_j\}$ such that $ \varphi(\tI_j) \in \cD(\tilde{N},A)$ and $\tI_j \subset I$ and $| \tI_j|\geq 1/(2M\tilde{N}^2)$ and $\sum | \tI_j| \geq (1-2dM^2\tilde{N}^{-1})  |I|$.
\end{lemm} 

\begin{proof} 
We just observe that the set of $t_k \in I$ such that $\l+e^{i2\pi \f(t)}=0$ with $\l \in \mathcal E(A)$ and $k \leq \tilde{N}$ consists of at most $d([M \tN^2 |I|]+2)$ points separated one from the other by at least $1/(M \tN^2)$. Excluding from $I$ the intervals $[t_k-M/\tN^3,t_k+M/\tN^3]$ leaves us with a collection of intervals of size greater than $1/(2M\tilde{N}^2)$ of total length $|I|- d([M \tN^2 |I|]+2)M/\tN^3   \geq (1-2dM^2\tilde{N}^{-1})  |I|$.

\end{proof}

Recall that { \begin{equation} \label{M1} \max(\|\f \|_{lip(I_0)},\|\p \|_{lip(I_0)}) \leq \frac{M}{2}, \quad \inf_{t\in I_0}\varphi'(t) \geq \frac{2}{M} \end{equation}}


Let $N_0\geq N_0(M)$ of Lemma \ref{exclusion} and define for $n\geq 1$,  $N_n=N_{n-1}^{\frac{3}{2}}$. 

{ Observe that Lemma \ref{exclusion} implies that if $\cA_n$ is a collection of intervals of sizes greater than $1/(2M N_n^2)$ and $\f_n$ and $\psi_{n}$ are functions 
satisfying (\ref{M1}) on $\cA_n$ with $M$ instead of $2M$ then there exists $\cA_{n+1}$ that is a collection of intervals with sizes greater than  $1/(2M N_{n+1}^2)$ such that $\f_{n}(\cA_{n+1}),\p_{n}(\cA_{n+1}) \subset \cD(N_{n+1}) $ and $\lambda(\cA_{n+1}) \geq (1-2dM^2 N_{n+1}^{-1}) \lambda(\cA_n)$.}

We now describe the inductive scheme that we obtain from an iterative application of Proposition \ref{mainprop}. At step $n$ we have $f_n=f_{\f_n}+ \Delta f_n$,$g_n=g_{\p_n}+ \Delta g_n$ defined for $t \in \cA_n$, with $\cA_{-1}=[0,1]$. We denote $\eps_{n,r}=\max(  \|\Delta f_n \|_{lip(\cA_n),r},\|\Delta g_n\|_{lip(\cA_n),r} )$. We obtain $h_n$ and $\f_{n+1}$ and $\p_{n+1}$ defined on $\cA_{n+1}$  such that 
\begin{align*} H_n f_n H_n^{-1}&= f_{\f_{n+1}} + {\Delta f_{n+1}} \\  H_n g_n H_n^{-1}&= g_{\p_{n+1}} + {\Delta g_{n+1}} \end{align*}
with ${\Delta f_{n+1}},{\Delta g_{n+1}} \in C^{lip(\cA_{n+1}),\infty}_0(I,\T^{d+1},\R^{d+1})$,
and if we denote $\xi_{n,r}=\|h_n\|_{lip(\cA_{n+1}),r+1}$ and  $\nu_n=\max(\| \f_{n+1}-\f_n \|_{lip(\cA_{n+1})},\| \p_{n+1}-\p_n \|_{lip(\cA_{n+1})}) $ we have from Proposition \ref{mainprop} that
\begin{align}
 \label{hn} \xi_{n,r} &\leq C_r  \gamma_nN_n^{\si}\eps_{n,r} \\
\label{phin} \nu_n& \leq \eps_{n,0} \\
\label{epsn} \eps_{n+1,r} &\le  C_r  \gamma_n N_n^{\si} \eps_{n,0}\eps_{n,r} + C_{r,r'}  \gamma_n N_n^{\si+r-r'} \eps_{n,r'}
\end{align}
with $\gamma_n=(1+S_n+\eps_{n,0})^\si$.

{ If during the induction we can insure that $\sum \eps_{n,0} <M/100$ we can conclude from (\ref{phin}) and the definition of $M$ that for all $n$,  $\f_n$ and $\psi_{n}$ satisfy 
 on $\cA_n$ the inductive condition 
$$({\rm C1}) \quad 2\max(\|\f_n \|_{lip(\cA_n)},\|\p \|_{lip(\cA_n)}) \leq M, \quad \inf_{t\in \cA_n}\varphi_n'(t) \geq \frac{1}{M}$$}

and Lemma \ref{exclusion} will insure that  $\cA_{n+1}$ is well defined and $\lambda(\cA_{n+1}) \geq (1-2M^2 N_{n+1}^{-1}) \lambda(\cA_n)$.
To be able to apply the inductive procedure we also have to check that $H_n$ is indeed invertible which is insured if during the induction we have 
$$({\rm C2}) \hspace{4cm}  \xi_{n,0}<\frac{1}{2}. \hspace{8cm} \hfill$$  We call the latter two conditions the inductive conditions.

The proof that the scheme (\ref{hn})--(\ref{epsn}) converges provided an adequate control on $\eps_{0,0}$ and $\eps_{r_0,0}$ for a sufficiently large $r_0$ is classical but we provide it for completeness. 

\begin{lemm} \label{kam} Let $\a=4\si+2$, $\beta=2 \sigma +1$, and $r_0=[8\si+5]$. If  $S_n,\xi_{n,r},\eps_{n,r}$ satisfy (\ref{hn})--(\ref{epsn}), there exists $\bar{N}_0(\si)$ such that if $N_0=\bar{N}_0 M$ and  
$$\eps_{0,0}\leq N_0^{-\a}, \quad \eps_{0,r_0}\leq N_0^{\beta}$$ then 
for any $n$  the inductive conditions (C1) and (C2) are satisfied and in fact 
$\eps_{n,0}\leq N_n^{-\a}$, $\xi_{n,0}\leq N_n^{-\si}$, and for any $s \in \N$, there exists $\bC_r$ such that $\max(\eps_{n,s},\xi_{n,s}) \leq \bC_s N_n^{-1}$. 
\end{lemm} 

\begin{proof} We first prove by induction that for every $n$, $\eps_{n,0}\leq N_n^{-\a}$ and $\eps_{n,r_0}\leq N_n^{\beta}$, provided  $\bar{N}_0(\si)$ is chosen sufficiently large.

Assuming the latter holds for every $i\leq n$,  the inductive hypothesis (C1) and (C2) can be checked up to $n$ immediately from (\ref{hn}) and (\ref{phin}). 
Now, (\ref{epsn}) applied with $r=0$ and $r'=r_0$ yields 
\begin{align*} \eps_{n+1,0} &\le  C_0 N_n^{\si}(2+M)^\si N_n^{-2\a} + C_{0,r_0} N_n^{\si-r_0} N_n^{\beta} \\
&\leq N_{n+1}^{-\a} \end{align*}
provided $\bar{N}_0(\si)$ is sufficiently large.

On the other hand, applying (\ref{epsn}) with $r'=r=r_0$ yields 
\begin{align*} \eps_{n+1,r_0} &\le  C_{r_0} N_n^{\si}(2+M)^\si N_n^{-\a} N_n^\beta  + C_{r_0,r_0} N_n^{\si} N_n^{\beta} \\
&\leq N_{n+1}^{\beta} \end{align*}
provided $\bar{N}_0(\si)$ is sufficiently large.

To prove the bound on $\eps_{n,s}$ we start by proving that for any $s$, there exist $\tilde{C}_s$ and  $n_s$ such that for $n\geq n_s$ we have that $\eps_{n,s}\leq \tilde{C}_s N_n^\beta$. Let indeed $n_s$ be such that $N_{n_s}^{-1/10} ((1+M)^\si C_s+C_{s,s}) <1$. Let $\tilde{C}_s$ be such that $\eps_{n_s,s} \leq \tilde{C}_s N_{n_s}^\beta$. We show by induction that $\eps_{n,s} \leq  \tilde{C}_s  N_n^\beta$ for every $n\geq n_s$. Assume the latter true up to $n$ and  apply (\ref{epsn}) with $r=r'=s$ to get 

\begin{align*} \eps_{n+1,s} &\le  C_{s} N_n^{\si}(1+M)^\si  N_n^{-\a} \eps_{n,s} + C_{s,s} N_n^{\si} \eps_{n,s} \\
&\leq N_{n}^{\si+1/10} \eps_{n,s} \\ &\leq  \tilde{C}_s  N_n^{\si+1/10+\beta} \leq  \tilde{C}_s N_{n+1}^\beta. \end{align*}

We will now bootstrap on our estimates as follows. Let $s'(s)=s+[\sigma+\beta+\frac{3}{2}(\si+1)]+1$, and define $\tilde{n}_s=\max(n_s,n_{s'})$. Let $\bar{C}_s$ be such that $\eps_{\tilde{n}_s,s} \leq \bar{C}_s N_n^{-\sigma-1}$. We will show by induction that for any $n \geq \tilde{n}_s$ we have that $\eps_{n,s} \leq \bar{C}_s N_n^{-\sigma-1}$. Indeed, apply   (\ref{epsn}) with $r=s$ $r'=s'$ to get 

\begin{align*} \eps_{n+1,s} &\le  \bar{C}_s C_{s} N_n^{\si}(1+M)^\si  N_n^{-\a} N_n^{-\si-1}  + C_{s,s'} \tilde{C}_{s'}  N_n^{\beta} N_n^{\si+s-s'} \\
&\leq \bar{C}_s N_{n+1}^{-\si-1} \end{align*}
if $n_s$ was chosen sufficiently large.

Finally, (\ref{hn}) yields that for $n\geq \tilde{n}_s$, $\xi_{n,s} \leq  C'_s N_n^{-1}$.

\end{proof}

\noindent{\it Proof of the main theorem.} 

The sets $\cA_n$ are decreasing and we let $\cA_\infty=\liminf \cA_n$.
The result of Lemma \ref{kam} implies that 
$$\lambda ( \cA_\infty )\geq \Pi (1-2M^2 N_{n+1}^{-1}) \geq 1-\eta$$ 
if $N_0\geq N_0(\eta)$. On $\cA_\infty$, $\f_n$ and $\p_n$ converge in the Lipschitz norm and 
the maps $H_n \circ \ldots \circ H_1,H_n^{-1} \circ \ldots \circ H_1^{-1}$ converge in the $C^{lip,\infty}$ norm to some $G,G^{-1}$ such that $G f_{\f} G^{-1} = f_{\f_\infty}$, $G g_{\p} G^{-1} = g_{\p_\infty}$, where $(\f_\infty,\p_\infty) = \lim_{n\to \infty} (\f_n,\p_n)$.

\section{Proof of Theorem \ref{main} in the case of higher dimensional elliptic factors, $d_2>1$} \label{generald2}

Define instead of the set $\cD(N,A)$ of Section \ref{section.inductive} the following 
\begin{multline*} \cD(N,A) = \{ \a \in \T^{d_2} \ / \ |\l + e^{i2\pi  (k,\a )} | \geq N^{-b}, \\ \forall \l \in \mathcal E(A), \forall  k \in \Z^{d_2}-\{0\},  \|k\|\leq N\} \end{multline*}
where $b=30d_2^2$. 
Instead of Lemma \ref{exclusion} we have the following more general statement.
\begin{lemm} \label{exclusiond} Let $\nu>0$. There exists $N_0(\nu,d_2)$ such that if $N>N_0$  and  if $I$ is an interval of size $1\geq |I|\geq 1/N^a$, $a=4d_2+20$, and if $\varphi : I \to \T^{d_2}$ satisfies a Pyartli  condition with constant $\nu$, then for $\tilde{N}=N^{3/2}$, there exists a union of disjoint intervals $\cU = \{ \tI_j\}$ such that $\tI_j \in \cD(\tilde{N},A)$ and $\tI_j \subset I$ and $| \tI_j|\geq 1/\tilde{N}^a$ and $\sum | \tI_j| \geq (1-\tilde{N}^{-1})  |I|$.
\end{lemm} 

\begin{proof} The proof is a direct consequence of the Pyartli condition and a repeated application of the intermediate value theorem. We just deal with case $\l=1$ the other cases being similar. More precisely, for any fixed $k$, $\|k\|\leq N$, after excluding $d_2$ intervals of size $1/N^{a}$ from $I$ we get that $|(k,\varphi')|\geq N^{-a(d_2+1)}$. After excluding $\mathcal O(N)$ intervals of size $N^{a(d_2+1)-b}$ we remain with intervals on which $\|(k,\varphi)\|\geq N^{-b}$. We then apply this procedure for every $k \in \Z^{d_2}$ such that $0<\|k\|\leq N$, then further eliminate all the intervals that are smaller than $\tilde{N}^{-a}$,  and finally observe that the remaining part of $I$ is a union of intervals satisfying the conditions of the lemma. 
\end{proof}

The effect of changing the exponent in the definition of $\cD(N,A)$ just modifies $\sigma(A,B)$ of Proposition \ref{mainprop} to make it $\sigma(A,B,d_2)$. This is because in (\ref{h-0}) the small divisor (in the case $|\lambda|=1$) becomes \newline $\frac{1}{|\l-e^{2\pi i (m,\vp(t))}|} \leq N^b$ if $m\in \Z^{d_2}$ is such that $|m|\leq N$. The rest of the proof of Proposition \ref{mainprop} is identical to the case $d_2=1$, except that everywhere the Lipschitz norm in the parameter direction should be replaced by the $C^{d_2}$ norm. If we assume WLOG that $\f$ satisfies an initial Pyartli condition with constant $\nu$, then similarly to what was done in the case $d_2=1$, we insure in the KAM scheme that a Pyartli condition with a fixed constant $\nu/2$ is satisfied by the functions $\varphi_n$, provided the control on the perturbation $\eps$ is sufficiently small.

\section{Proof of Theorem \ref{main2}} \label{proofmain2}

Let $A,B,\a,\beta$ and $f,g$ be as in the statement of Theorem \ref{main2}. 
Let us momentaneously assume that $\a \in \text{DC}(\tau,\gamma,A)$ that is $ |\l - e^{i2 \pi (k,\a)} |> \frac{\gamma}{|k|^\tau}$ for every non zero vector $k \in \Z^{d_2}$ and every $\l \in \mathcal E(A)$.
This clearly plays a similar role to $\varphi(t) \in \cD(A)$ and the same proof as that of Proposition \ref{mainprop} yields a conjugacy $ H=\id + h$ such  that 
\begin{equation}\label{nonlin2}
\begin{aligned} H \circ f  &= (\tilde{f}_0 + \widetilde{\Delta f}) \circ H \\
H \circ g  &= (\tilde{g}_0 + \widetilde{\Delta g}) \circ H \end{aligned}
 \end{equation}
with $  \tilde{f}_0= A \times R_{\tilde{\a}},  \tilde{g}_0= B\times R_{\tilde{\beta}}$ and  $h, \widetilde{\Delta f}, \widetilde{\Delta g}$ satisfy estimates as in Proposition \ref{mainprop}. Now, the fact that $(\rho_{\mu_1}(f),\rho_{\mu_2}(g))=(\a,\beta)$ implies that $(\rho_{H_* \mu_1}(H \circ f \circ H^{-1}),\rho_{H_*\mu_2}(H \circ g \circ H^{-1}))=(\a,\beta)$. In conclusion we can replace $\tilde{f}_0,\tilde{g}_0$ by $A\times R_\a, B\times R_\beta$ in (\ref{nonlin2}) and include $\tilde{\a}-\a$,$\tilde{\b}-\b$ inside the error terms without changing the quadratic nature of the estimates. 

For the general case $(\a,\beta) \in \text{SDC}(\tau,\gamma,A,B)$ one cannot use just one of the frequencies $\a$ or $\beta$ to solve the linearized equations of (\ref{simpleconj}). Indeed, both $\a$ and $\beta$ may be Liouville vectors and the small divisors that appear in (\ref{h-0}) may be too large. Actually the linearized system  (\ref{simpleconj}) will not be solved as in Claim 1 but just up to an error term that is quadratic as in Lemma \ref{splittinglemma}. The idea goes back to Moser \cite{M} who observed that 
if for each $m$ one of the small divisors ${\l-e^{2\pi i m\a}}$  or ${\mu-e^{2\pi i m\beta}}$ is not too small, as stated in the SDC condition, then the relation implied by the commutation
(\ref{simplecom}) 
 \begin{align*}
(\la w-w\circ f_\f)-(\mu v-v\circ g_\p)=\phi
\end{align*}
insures that (\ref{simpleconj}) can be solved up to an error term of the order of $\phi$, that is a quadratic error term as in \eqref{quad}.

 The rest of the proof of Theorem \ref{main2} is identical to that of Theorem \ref{main3}. $\hfill \Box$

\section{Appendix}

In the Appendix we give references and proofs for the estimates used in the proofs of Lemma \ref{splittinglemma} and Proposition \ref{mainprop}.

\subsection{Convexity estimates}

\begin{prop} \label{hadamard} 
Let $f,g \in  C^{lip,\infty}(I, \T^{d},\R)$. Then

\begin{itemize}
\item[(i)]
$$
||f||_{lip(I),s}\le C_{s_1,s_2}||f||_{lip(I),s_1}^{a_1}||f||_{lip(I),s_2}^{a_2}$$
for all non-negative numbers $a_1,a_2,s_1,s_2$ such that
$$
a_1+a_2=1,\quad s_1a_1+s_2a_2=s.$$

\item[(ii)]
$$
||fg||_{lip(I),s}\ \le\ C_s(||f||_{lip(I),s}||g||_{lip(I),0}+||f||_{lip(I),0}||g||_{lip(I),s})$$
for all non-negative numbers $s$.
\end{itemize}
\end{prop}
\begin{proof}
(i) One way to show interpolation estimates in the scale of $C^{lip, s}$ norms is to derive them from the existence of smoothing operators and from the norm inequalities for the smoothing operators. This is done  in \cite{Zehnder} for spaces $C^{\alpha, s}$ where $0<\alpha\le 1$, which includes the case of $C^{lip, s}$. Another elementary proof for interpolation without going through smoothing operators can be found in \cite{OdlL}. 

(ii) Immediate corollary of the interpolation estimates is the following fact:
$$\|f\|_{lip(I), i}\|g\|_{lip(I), j}\le C (\|f\|_{lip(I), k}\|g\|_{lip(I), l}+\|f\|_{lip(I), m}\|g\|_{lip(I), n})$$
if $(i,j)$ lies on the line segment joining $(k, l)$ and $(m, n)$. (See Corollary 2.2.2. in \cite{Ham}).
The statement (ii) in the Proposition follows from this by using the product rule on derivatives (see Corollary 2.2.3. in \cite{Ham}) and the following inequality: 
\begin{equation*}
\begin{aligned}
Lip (fg)&=\sup_{x\ne y} \frac{|(fg)(x)-(fg)(y)|}{|x-y|} \\
&\le \sup ( \frac{|f(x)-f(y)||g(x)|}{|x-y|} + \frac{|g(x)-g(y)||f(y)|}{|x-y|})\\
&\le L_f\|g\|_0+\|f\|_0L_g
\end{aligned}
\end{equation*}
where $L_f$ and $L_g$ are Lipshitz constants for $f$ and $g$, respectively.
\end{proof}

\subsection{Composition}

\begin{prop} \label{annexe.compose} 
Let $f,g \in  C^{lip,\infty}(I, \T^{d+1},\R^{d+1})$.Then
\begin{itemize}

\item[(i)]
$h(x)=f(x+g(x))-f(x)$ verifies
$$ \| h\|_{lip(I),s}
\leq C_s(\| f\|_{lip(I),0} \|g\|_{lip(I),s+1}+ \|  f\|_{lip(I),s+1}\|g\|_{lip(I),0}).$$ 

\item[(ii)]
$k(x)=f(x+g(x))-f(x)-Df g(x)$ verifies
$$ \| k\|_{s}\leq 
C_s (\|  f\|_{lip(I),0} \|g\|_{lip(I),s+2}+ \|  f\|_{lip(I),s+2}\|g\|_{lip(I),0}) $$
\end{itemize}

\end{prop}
\begin{proof} 
In the proof we shorten the notation $\|\cdot\|_{lip(I), s}$ to $\|\cdot\|_{lip, s}$.

(i) It suffices to prove the estimates for the coordinate functions of $f$, so in what follows we assume that $f$ denotes a coordinate function of $f$. Let $D_i^1$ denote partial derivation in one of the basis directions and let $g_j$ denote coordinate functions of $g$. Since $D_i^1 h=D_i^1(f(x+g(x))-f(x))=\sum_j D_j^1fD_i^1g_j$, we can apply part (ii) of the previous proposition to $D_j^1fD_i^1g_j$:

\begin{equation*}
\begin{aligned}
\|D^1 h\|_{lip, s}&\le C\max_j \|D_j^1fD_i^1g_j\|_{lip, s} \\
&\le C_s\max_j(\|D_j^1 f\|_{lip, s}\|D_i^1g_j\|_{lip, 0}+\|D_j^1 f\|_{lip, 0}\|D_i^1g_j\|_{lip, s})\\
&\le C_s\max_j (\|f\|_{lip, s+1}\|g_j\|_{lip, 1}+\|f\|_{lip, 1}\|g_j\|_{lip, s+1})\\
&\le C'_s(\|f\|_{lip, s+2}\|g\|_{lip, 0}+\|f\|_{lip, 0}\|g\|_{lip, s+2})
\end{aligned}
\end{equation*}
where we invoked part (ii) of the previous proposition again to obtain the last line of estimates above. 
Since for the $lip, 0$-norm we have:
$$\|h\|_{lip, 0}=\|f(x+g(x))-f(x)\|_{lip, 0}\le L_f\|g\|_0\le \|f\|_{lip, 0}\|g\|_{lip, 0},$$
the claim follows. 

(ii) Again by reducing to coordinate functions we look at one coordinate function of $k$ and $f$ (which we denote by $k$ and $f$ as well), so we have
$k=f(x+g(x))-f-\sum_i D_i^1f g_i$, where $D_i^1$ denotes $\partial/\partial x^i$.
Then:
$D_j^1k=-\sum_iD_j^1D^1_i f g_i$, where $g_i$ denotes coordinate functions of $g$. This implies (by using (ii) of Proposition \ref{hadamard}) the following estimate for the first derivatives:
\begin{equation*}
\begin{aligned}
\|D_j^1 k\|_{lip, s}&\le \sum_i \|D_j^1D^1_i f g_i\|_{lip, s} \\
&\le C_s(\|D_j^1D^1_i f\|_{lip, s}\|g_i\|_{lip, 0}+\| D_j^1D^1_i f\|_{lip, 0}\|g\|_{lip, s})\\
&\le C_s(\|f\|_{lip, s+2}\|g_i\|_{lip, 0}+\|f\|_{lip, 2}\|g\|_{lip, s})\\
&\le C'_s(\|f\|_{lip, s+2}\|g\|_{lip, 0}+\|f\|_{lip, 0}\|g\|_{lip, s+2})
\end{aligned}
\end{equation*}
For the $lip, 0$-norm we have:
$$\|k\|_{lip, 0}\le L_f\|g\|_0+\max_i\{\|D_i^1 fg_i\|_{lip, 0}\}\le C\|f\|_{lip, 1}\|g\|_{lip, 0}$$
which together with the estimates above implies the claim.
\end{proof}
\subsection{Inversion}

\begin{prop} \label{annexe.inverse} 
Let $h \in  C^{lip,\infty}(I, \T^{d+1},\R^{d+1})$  and assume that
$$\|h\|_{lip(I),1}\leq \frac{1}{2}$$
Then
$$f: \T^{d+1} \to \T^{d+1} , x\mapsto  H(x)=x+ h(x)$$
is invertible and if we write $H^{-1}(x)=x+\bar h (x)$ then
$$\|\bar h\|_{lip(I),s}  \leq C_s \|h\|_{lip(I),s} $$
for all $s\in\N$.
\end{prop}
\begin{proof}
For $C^s$ norms this is proved for example in Lemma 2.3.6. in \cite{Ham}. The proof uses induction and interpolation estimates, and it is general to the extent that it applies to any sequence of norms on $C^\infty$ which satisfy interpolation estimates. Thus the claim follows from part (i) of the Proposition \ref{hadamard} and Lemma 2.3.6. in \cite{Ham}.\end{proof}

\end{document}